\documentclass[12pt,a4paper]{article}
\usepackage{amsfonts}
\usepackage{mathrsfs}
\usepackage{color}
\usepackage{stmaryrd}
\usepackage{amsmath}
\usepackage{amssymb}
\usepackage{bbm}
\usepackage{graphicx}
\usepackage{enumerate}
\usepackage{theorem}
\usepackage{booktabs}
\usepackage{multirow}
\usepackage{booktabs}
\usepackage{marvosym}
\usepackage{multirow}
\usepackage{hyperref}

\makeatletter
\def\tank#1{\protected@xdef\@thanks{\@thanks
        \protect\footnotetext[0]{#1}}}
\def\bigfoot{

    \@footnotetext}
\makeatother

\newcommand{\ea}{\end{array}}
\newtheorem{theorem}{Theorem}[section]
\newtheorem{proposition}{Proposition}[section]
\newtheorem{corollary}{Corollary}[section]
\newtheorem{lemma}{Lemma}[section]
\newtheorem{definition}{Definition}[section]
\newtheorem{remark}{Remark}[section]

{\theorembodyfont{\rmfamily}
}

\newenvironment{proof}{Proof.}

\begin{document}

\title{\Large \bf Stochastic curve shortening flow driven by a  transport-type pure jump L\'evy noise}

\author{{Xiaotian Ge}$^a$\footnote{E-mail:gxt@mail.ustc.edu.cn (X.T. Ge)} ~~~{Shijie Shang}$^a$\footnote{E-mail:sjshang@ustc.edu.cn (S.J. Shang)} ~~~{Weina Wu}$^{b}$\footnote{E-mail:wuweinaforever@163.com (W.N. Wu)} ~~~ {Jianliang Zhai}$^a$\footnote{E-mail:zhaijl@ustc.edu.cn (J.L. Zhai)}
\\
 \small  a.  School of Mathematical Sciences,
University of Science and Technology of China,\\
\small Hefei, Anhui 230026, China.\\
 \small  b. School of Economics, Nanjing University of Finance and Economics,\\
 \small  Nanjing, Jiangsu 210023, China.
}
\date{}
\maketitle

\begin{center}
\begin{minipage}{130mm}
{\bf Abstract:} We study the existence and uniqueness, the regularity, and the long-time behavior of strong solutions to stochastic curve shortening flow  driven by a transport-type pure jump L\'{e}vy noise. To obtain the existence and uniqueness of strong solutions, we transform the equation into its equivalent It\^{o}-type stochastic partial differential equation via a transport equation, and apply the monotone method with Lyapunov-type conditions. The obstacles to investigate the long-time behavior are the weak dissipativity and singularity inherent in the equation. To this end, we establish an improved regularity and prove that these solutions converge pathwise to zero at an exponential rate.

\vspace{3mm}
{\bf Keywords:} curve shortening flow; transport-type pure jump L\'{e}vy noise; well-posedness; regularity; long-time behavior

\vspace{3mm}
{\bf AMS 2020 Mathematics Subject Classification:} 60H15; 60G51
\end{minipage}
\end{center}

\section{Introduction}

In this paper, we are concerned with the following stochastic curve shortening flow driven by a transport-type pure jump L\'{e}vy noise:
\begin{equation}\label{s1eq2}
du(t,x) = \frac{\partial_{xx} u(t,x)}{1 + (\partial_x u(t,x))^2} dt + \epsilon \partial_x u(t,x) \diamond dL_t, \quad (t,x) \in \Bbb{R}_+\times \mathbb{T},
\end{equation}
where $\partial_x$ denotes the weak derivative operator with respect to $x$, $\epsilon$ is a positive constant which represents the intensity of noise, $\diamond $ denotes the Marcus integral, $\Bbb{R}_+=[0, \infty)$, $\mathbb{T} = \mathbb{R}/\mathbb{Z} \cong [0,1]$ is the one-dimensional torus. $\{L_t\}_{t \geq 0}$ is an $\mathbb{R}$-valued pure jump L\'{e}vy process defined on $(\Omega, \mathcal{F}, \{\mathcal{F}_t\}_{t \geq 0}, \mathbb{P})$, which is a complete probability space equipped with a filtration $\{\mathcal{F}_t\}_{t \geq 0}$ satisfying the usual conditions. Please see Section \ref{Main result} for the detailed exposition of \eqref{s1eq2}.

The curve shortening flow is the one-dimensional mean curvature flow, evolving a smooth planar curve by moving each point normally inward at a speed proportional to its curvature. As the simplest gradient flow of the surface area energy, mean curvature flow has many applications, for instance, it arises as an asymptotic limit of the Allen-Cahn model for phase boundary motion in binary alloys, see e.g., \cite{MR3155251,MR2815949,MR1931534} and the references therein. In addition, the research methods and results for it can provide certain insight into some related complex problems. Mathematically, curve shortening flow corresponds to the graph formulation of $(1+1)$-dimensional mean curvature flow, given by \eqref{s1eq2} with $\epsilon=0$. It is a fundamental geometric evolution that serves as a paradigmatic model for studying singularity formation, gradient flows and topological changes, driving developments in geometric analysis. Besides, its theoretical importance is complemented by practical utility, modeling interface dynamics in materials science, and supporting applied tasks in image processing and computer vision, see e.g., \cite{MR1225209,CKS1997,G2020,MR1343443,MB2020} and the references therein for further information.

Up to our knowledge, stochastic mean curvature flow was first proposed by Kawasaki and Ohta in \cite{10.1143/PTP.67.147} (see also \cite{MR1959710}) as the following model (in graph form):
\begin{equation}\label{s1eq4}
du(t,x)=\sqrt{1+(\partial_x u(t,x))^2}\,{\rm div}\left(\frac{\partial_x u(t,x)}{\sqrt{1+(\partial_x u(t,x))^2}}\right)dt+
\phi(x,u(t,x))\sqrt{1+(\partial_x u(t,x))^2}\circ dB_t.
\end{equation}
Since then, the following two stochastic curve shortening flow models have been widely discussed, i.e.,
\begin{equation}\label{s1eq0.2}
du(t,x)=\frac{\partial_{xx} u(t,x)}{1 + (\partial_x u(t,x))^2}dt+dW_t
\end{equation}
and
\begin{equation}\label{s1eq0.3}
du(t,x)=\frac{\partial_{xx} u(t,x)}{1 + (\partial_x u(t,x))^2}dt+\sqrt{1+(\partial_x u(t,x))^2}\circ dB_t.
\end{equation}
Here $\{W_t\}_{t\geq0}$ is a $Q$-Wiener process, $Q$ is a Hilbert-Schmidt operator, $\circ$ denotes the Stratonovich integral, and $\{B_t\}_{t\geq0}$ is a one-dimensional Brownian motion. In \cite{MR2888287,MR2996423}, the well-posedness and long-time behavior of \eqref{s1eq0.2} and \eqref{s1eq0.3} were studied under the variational framework of stochastic partial differential equations (SPDEs). The well-posedness of \eqref{s1eq0.3} has also been considered by using the stochastic variational inequality solutions framework and the entropy solutions framework,  see \cite{MR3605963} and \cite{MR4089785}, respectively. In \cite{MR4674013}, \eqref{s1eq0.3} was treated under the stochastic viscosity solutions framework and its pathwise long-time behavior was also provided. We refer the readers to \cite{MR4221662,MR2888287,MR4674013,MR4692883,MR3651056} for more surveys on \eqref{s1eq4} and its related topics, and \cite{MR3205643,MR3053478,MR4646815} for more results on the well-posedness and ergodicity for equations including \eqref{s1eq0.2}.

Although there have been some studies on stochastic curve shortening flow driven by continuous noise as discussed above,  there is no result on stochastic curve shortening flow driven by jump noise so far. To investigate this, we employ L\'evy noise due to its flexibility in modeling non-Gaussian, heavy-tailed, and long-range dependent phenomena (see e.g., \cite{MR2356959} for more knowledge about L\'evy noise). Specifically, in this paper, we focus on \eqref{s1eq2} with a linear, transport-type pure jump L\'evy noise. For systems with continuous noise,  from
a physical point of view, the Stratonovich integral  (as in \eqref{s1eq4} and \eqref{s1eq0.3}) is a natural choice because it respects the chain rule and is invariant under changes of coordinates, while the Marcus canonical form (cf. \cite{KPP-95,Marcus-78} for more details) is specifically designed as a natural extension of the Stratonovich calculus to the jump processes. In our case, we use the Marcus canonical integral, which consistently preserves the chain rule and captures the geometry of jumps. Furthermore, it is the extension of the Wong-Zakai principle to discontinuous settings, i.e., if a jump process (e.g., L\'evy process or semi-martingale with jumps) is approximated by a sequence of appropriately converging smooth processes, the solutions of the approximating ordinary differential equations converge precisely to the solution of the corresponding Marcus canonical stochastic differential equation. Therefore, from both theoretical and applied perspectives, it is significant to study equations driven by jump noises in Marcus canonical form. For recent progress on its applications to SPDEs, see works on Landau-Lifshitz-Gilbert equations (cf. \cite{BM2019}),  nonlinear Schr\"odinger equations (cf. \cite{BLZ2021,ZLZ2026}), 2D constrained Navier-Stokes equations (cf. \cite{MA2021}),  2D Euler equations (cf. \cite{LT2025}) and transport equations (cf. \cite{FPR2025}).

The primary objectives of our work (cf. Theorem \ref{c2th2} for the main result) are twofold. Firstly, we establish the existence and uniqueness of strong solutions to \eqref{s1eq2}. Secondly, we prove that these solutions converge pathwise to zero at an exponential rate.

To establish the existence and uniqueness of solutions, we adopt the monotone method with Lyapunov-type conditions. To be precise, we use the methods and results from \cite[Theorem~1.2]{MR3158475} for fluid equations driven by jump noise, and the approach used in \cite[Theorem~2.3]{MR2888287} for degenerate nonlinear stochastic evolution equations driven by Brownian motion within the Krylov-Rozovski\u{\i} variational framework, in which the coercivity assumption is replaced by a Lyapunov-type condition. However, due to the features of \eqref{s1eq2} itself, especially the noise term with Marcus-type integral, we need  to apply different techniques and deal with some estimations that are not encountered in \cite{MR3158475} or \cite{MR2888287}. In the proof, we first transform \eqref{s1eq2} into its equivalent It\^{o}-type SPDE (cf. \eqref{c2eq6}) via a transport equation (cf.  \eqref{c2eq1}). The task then reduces to the proof of well-posedness for \eqref{c2eq6} (cf. Proposition \ref{s6th1}), but with more complicated noise terms than the standard It\^o noise and with an additionally correction term. Following from the strategies of \cite{MR3158475} and \cite{MR2888287}, part of our main efforts here are the estimations of such terms in verifying the hemicontinuity, local monotonicity, Lyapunov condition and boundedness, where the properties of the  transport equation (cf. Lemma \ref{c2th1}) also play crucial roles, see for instance, \eqref{s6eq10}, \eqref{s6eq19}, \eqref{s6eq24}, etc.

The obstacles in investigating the long-time behavior of solutions to (stochastic) curve shortening flow arise primarily from the weak dissipativity (or certain degeneracy) and singularity inherent in the equation itself, see e.g., Remark \ref{dissipativity} for more details. In \cite{MR2888287}, they directly analyze the equation but have to restrict the noise intensity to satisfy $\epsilon\leq\sqrt{2}$. Different from \cite{MR2888287}, our method establishes an improved regularity, and has no restriction on the noise intensity. This is achieved by a four-step proof, which are given as Lemmas \ref{s6th3} $\sim$ \ref{s6th5} and Proposition \ref{s6th6+1}. More specifically, firstly, we obtain a pathwise identity (cf. \eqref{s6eq26}) for the $W^{1,2}(\mathbb{T})$-norm  of the Galerkin approximating equations to \eqref{c2eq6} (which is the equivalent It\^{o}-type form of \eqref{s1eq2}). Secondly, we prove that the solutions of the approximating equations converge strongly to the solution of \eqref{c2eq6} in $L^2(\mathbb{T})$. Thirdly, by constructing a monotone function and applying pseudo-monotonicity techniques, we are able to pass to the limit and obtain the corresponding pathwise $W^{1,2}(\mathbb{T})$-norm inequality for the solution of \eqref{c2eq6}. Finally, we get a pathwise $W^{2,1}(\mathbb{T})$-norm estimate for the solution to \eqref{c2eq6} (hence also \eqref{s1eq2}) by means of the established $W^{1,2}(\mathbb{T})$-norm inequality.
The essential tools we use here are the properties of maximal monotone operators and a certain energy-conservation property of Marcus noise, which enable us to gain the pathwise properties and further derive the pathwise convergence.
Once the improved regularity holds, the long-time behavior of  solutions to \eqref{c2eq6} (hence also \eqref{s1eq2}) follows by using appropriate Sobolev embeddings and properties of solutions to the transport equation \eqref{c2eq2}. We emphasize that, this is the first result on the pathwise exponential decay for stochastic curve shortening flow.

The rest of this paper is organized as follows. In Section \ref{Main result}, we give a detailed exposition of the curve shortening flow driven by a transport-type pure jump L\'evy noise in the Marcus canonical form, and state the main result: the well-posedness, regularity and long-time behavior of \eqref{s1eq2}.  Section \ref{wellposedness} is devoted to the proof of well-posedness. The regularity and long-time behavior are then established in Section \ref{longtime}.

Throughout the paper, for some space $\mathbb{H}$, we denote by $\mathcal{B}(\mathbb{H})$ the Borel $\sigma$-algebra in $\mathbb{H}$, $||\cdot||_{\mathbb{H}}$ the norm in $\mathbb{H}$,  $L^\infty([0,\infty);\mathbb{H})$ the space of all $\mathbb{H}$-valued  essentially bounded measurable functions on $[0,\infty)$, $L^2([0,\infty);\mathbb{H})$ the space of all $\mathbb{H}$-valued square-integrable functions on $[0,\infty)$, $L^2([0,T];\mathbb{H})$ the space of all $\mathbb{H}$-valued square-integrable functions on $[0,T]$, $L^2([0,T]\times\Omega;\mathbb{H})$ the space of all $\mathbb{H}$-valued square-integrable functions on
$[0,T]\times\Omega$, $L^1([0,T]\times\Omega;\mathbb{H})$ the space of all $\mathbb{H}$-valued Bochner-integrable functions on $[0,T]\times\Omega$. We denote by $L^2(\mathbb{T})$ the Lebesgue space on $\mathbb{T}$, and $L^2([0,T]\times\mathbb{T})$ the space of all $\Bbb{R}$-valued square-integrable functions on
$[0,T]\times\mathbb{T}$. Let $W^{1,2}(\mathbb{T})$ and $W^{2,1}(\mathbb{T})$ denote the well-known Sobolev spaces.
For notational convenience, we set $H^1(\mathbb{T}): = W^{1,2}(\mathbb{T})$, and let $H^{-1}(\mathbb{T})$ be the dual space of $H^1(\mathbb{T})$ with respect to the inner product on $L^2(\mathbb{T})$.

\section{Curve shortening flow driven by transport-type noise in Marcus canonical form}\label{Main result}
\setcounter{equation}{0}
 \setcounter{definition}{0}


Let $\{L_t\}_{t\geq0}$ be a real-valued pure jump L\'evy noise of the following form:
$$L_t=\int_0^t \int_{\{|z|\leq1\}} z\, \tilde{N} (dzds)+\int_0^t \int_{\{|z|>1\}}z\, N(dzds),\quad\forall t\geq0.
$$
Here $N$ is a time homogeneous Poisson random measure  from $\mathcal{B}(\Bbb{R})\otimes\mathcal{B}(\mathbb{R}_+)\otimes\mathcal{F}$ to $\mathbb{N}\cup\{0,\infty\}$ with the intensity measure $\nu$, and $\nu$ is a $\sigma$-finite measure on $(\Bbb{R},\mathcal{B}(\Bbb{R}))$ such that  $\nu(\{0\})=0$ and
$\int_{\Bbb{R}}(|z|^2\wedge1)\,\nu(dz)<\infty,$
$\tilde{N}(dzds):=N(dzds)-\nu(dz)ds$ is the compensated Poisson random measure.

Recall \eqref{s1eq2} in the introduction, for simplicity, from now on, we write $u(t):=u(t,x)$, i.e.,
\begin{equation}\label{c2eq1}
\begin{cases}
du(t)=\dfrac{\partial_{xx} u(t)}{1+(\partial_x u(t))^2}dt+\epsilon\partial_x u(t)\diamond dL_t,\quad \forall t>0, \\
u(0)=u_0.
\end{cases}
\end{equation}
Note that \eqref{c2eq1} can be transformed to a SPDE of the It\^{o} form. To be precise, consider the following PDE:
\begin{equation}\label{c2eq2}
\begin{cases}
\dfrac{\partial\Phi(\theta,z,u)}{\partial \theta}=\epsilon z \partial_x \Phi(\theta,z,u),\quad \forall \theta\in[0,1],\\
\Phi(0,z,u)=u\in H^1(\mathbb{T}),
\end{cases}
\end{equation}
where $\epsilon$ is the same as the one in \eqref{c2eq1}, $z\in\mathbb{R}$ is the jump size of $\{L_t\}_{t\geq0}$. It is easy to see that \eqref{c2eq2} is a transport equation and satisfies the following lemma.

\begin{lemma}
\label{c2th1} Let $u\in H^1(\mathbb{T})$, then \eqref{c2eq2} admits a unique weak solution such that for all $(x,\theta)\in \mathbb{T}\times [0,1]$, $\Phi(\theta,z,u)(x)=u(x+\epsilon z \theta)$,  and $\partial_{\theta\theta}\Phi(\cdot,z,u)=(\epsilon z)^2\partial_{xx}\Phi(\cdot,z,u)\in L^2([0,1];H^{-1}(\mathbb{T}))$.
\end{lemma}

Let $u\in L^2(\mathbb{T})$, denote by $[u]:=\int_\mathbb{T} u(r)\,dr$ the mean value of $u$ on $\mathbb{T}$. From now on, take
$$H=\dot{L}^2(\mathbb{T}):=\left\{u\in L^2(\mathbb{T}):\,[u]=0\right\},\quad\quad V=\dot{H}^1(\mathbb{T}):=\left\{u\in H^1(\mathbb{T}):\,[u]=0\right\},$$
and endow the spaces with the following inner products, respectively,
$$\left<u,v\right>_H=\int_\mathbb{T} u\cdot v \,dx,\quad \forall u,v\in H,\quad\quad\quad\left<u,v\right>_V=\int_\mathbb{T} \partial_x u\cdot \partial_x v \,dx,\quad \forall u,v\in V.$$

For any $k\in\mathbb{Z}:=\{0,\pm 1,\pm2,\cdots\}$, let
\begin{eqnarray}\label{ek}
 e_k:=e_k(x)=e^{2\pi {\rm i}kx}, ~ \text{with}~~ {\rm i}=\sqrt{-1},\quad \forall x\in \mathbb{T},
\end{eqnarray}
 then $\{e_k\}_{k\in\mathbb{Z}\backslash\{0\}}$ is an orthonormal basis of $H$, which is also an orthogonal basis of $V$. Take $H^n:={\rm span}\{e_k,0<|k|\leq n\}$, then according to Lemma \ref{c2th1}, we know that if $u\in H^n$, one has $\Phi(\theta,z,u)\in H^n$ for all $(z,\theta,\epsilon)\in\mathbb{R}\times[0,1]\times(0,\infty)$.

Let $\left<\cdot,\cdot\right>_{L^2}$ denote the inner product on $L^2(\mathbb{T})$.
If we identify the Hilbert space $H$ with its dual space $H^*$ by the Riesz representation, then we obtain a Gelfand triple $V\subset H\subset V^*$.
Define an operator $A: V\to V^*$ such that for all $u,v\in V$,
$$\ _{V^*}{\left<Au,v\right>}{_{V}}
=-\left<\arctan \partial_x u,\partial_x v\right>_{L^2}.$$

\begin{remark}\label{dissipativity}
Note that for $u, v\in H^1(\mathbb{T})$, one has
$$
_{H^{-1}(\mathbb{T})}{\langle Au-Av, u-v\rangle}_{H^1(\mathbb{T})}
=
-\int_{\mathbb{T}}
\bigl(\arctan(\partial_x u)-\arctan(\partial_x v)\bigr)
\bigl(\partial_x u-\partial_x v\bigr)dx
\leq 0,
$$
which implies the weak dissipativity (or certain degeneracy) of \eqref{c2eq1} (also \eqref{s1eq2}).
In particular, for $u\in H^1(\mathbb{T})$, one has $_{H^{-1}(\mathbb{T})}{\langle Au, u\rangle}_{H^1(\mathbb{T})}
=-\left<\arctan \partial_x u,\partial_x u\right>_{L^2}\leq 0,$ which indicates the singularity of \eqref{c2eq1} (also \eqref{s1eq2}) in the sense that, unlike the usual Laplace operator, it does not provide any direct gain of regularity for the solution.
\end{remark}

\begin{remark}
Since $\partial_x v$ and $\arctan \partial_x u$ might not be elements of $H$, we used the inner product of $L^2(\mathbb{T})$ as defined above. To avoid confusion, in the following, $\left<\cdot,\cdot\right>_H$ and $\|\cdot\|_H$ will be restricted to elements that are confirmed to be in $H$.
\end{remark}

By Lemma \ref{c2th1} and the definition of Marcus integral, \eqref{c2eq1} can be transformed  into the following It\^{o} form:
\begin{equation}\label{c2eq6}
\left\{
\begin{aligned} du(t)=&\,Au(t)dt+\int_{\{|z|\leq1\}}\Big( \Phi(1,z,u(t-))-u(t-)\Big)\tilde{N}(dzdt)\\
&+\int_{\{|z|>1\}} \Big(\Phi(1,z,u(t-))-u(t-)\Big) N(dzdt)\\
&+\int_{\{|z|\leq1\}} \Big(\Phi(1,z,u(t))-u(t)-\epsilon z\partial_x u(t)\Big)\nu(dz)dt,\quad \forall t>0,\\
u(0)=&\,u_0,
\end{aligned}
\right.
\end{equation}
where the second and third terms on the right-hand side are the noise terms, and the fourth term is usually called the correction term.

\begin{definition}\label{s2 df1}
  Let $u_0\in V$. An $H$-valued c$\grave{a}$dl$\grave{a}$g $\mathcal{F}_t$-adapted process $u:=\{u_t\}_{t\geq0}$ is called a strong solution to \eqref{c2eq6}, if \\
{\rm(i)} there exists a $dt\otimes\mathbb{P}$-equivalent class $\hat{u}$ of $u$ such that $\hat{u}\in L^2_{loc}([0,\infty);V)$, $\mathbb{P}$-a.s.;\\
{\rm (ii)} for all $t\in[0,\infty)$, the following equality
\begin{equation}\label{c2eq61}
\begin{aligned} u(t)=&u_0+\int_0^tA\bar{u}(s)ds+\int_0^t\int_{\{|z|\leq1\}}\Big( \Phi(1,z,\bar{u}(s))-\bar{u}(s)\Big)\tilde{N}(dzds)\\
&+\int_0^t\int_{\{|z|>1\}} \Big(\Phi(1,z,\bar{u}(s))-\bar{u}(s)\Big) N(dsdt)\\
&+\int_0^t\int_{\{|z|\leq1\}} \Big(\Phi(1,z,\bar{u}(s))-\bar{u}(s)-\epsilon z\partial_x \bar{u}(s)\Big)\nu(dz)ds,\ \text{holds\ in}\ V^*,\ \mathbb{P}\text{-}a.s.,
\end{aligned}
\end{equation}
where $\bar{u}$ is a $V$-valued progressively measurable $dt\otimes\mathbb{P}$ version of $\hat{u}$.
\end{definition}

\begin{remark}
Definition \ref{s2 df1} is motivated by \cite[Definition 1.1]{MR3158475}, in which the authors give the definition of strong solutions for SPDE with locally monotone coefficients driven by L\'evy noise. Here $\hat{u}$ and $\bar{u}$ are introduced precisely because they are required by the definition of Bochner integral as explained in \cite[Remark 1.1]{MR3158475}.
\end{remark}

Below is the main result on the well-posedness and long-time behavior of \eqref{c2eq6} (hence also \eqref{c2eq1}).

\begin{theorem} \label{c2th2}
Let $u_0\in V$. Then there exists a unique strong solution $u$ to \eqref{c2eq6} in the sense of Definition \ref{s2 df1}. Moreover, $u\in L^2([0,\infty); W^{2,1}(\mathbb{T}))\cap
L^\infty([0,\infty);V)$, $\mathbb{P}$\text{-}a.s., and   $||u(t)||_V^2\to0$ with exponential rate in the sense that there exists a constant $k_0>0$ (independent of $t$ and $\omega$) such that
\begin{equation}\label{c2eq9}
\lim_{t\to\infty}e^{k_0 t}||u(t)||^2_V=0, \ \mathbb{P}\text{-}a.s.
\end{equation}
\end{theorem}

\begin{remark}
In most works on the existence and uniqueness of solutions to SPDEs (see e.g., \cite{MR3158475,MR2888287,LT2025}), one can only obtain the local boundedness $u\in L^2_{\mathrm{loc}}([0,\infty); W^{2,1}(\mathbb{T}))\cap L^\infty_{\mathrm{loc}}([0,\infty);V)$, $\mathbb{P}$-a.s.,
i.e., for any $T\in[0,\infty)$, $\int_0^T\|u(t)\|^2_{W^{2,1}(\mathbb{T})}dt+\sup_{t\in[0,T]}\|u(t)\|_{V}<\infty$, $\mathbb{P}$-a.s. However, in our work, we can get the strong solution $u$ to \eqref{c2eq6} satisfying $
u\in L^2([0,\infty); W^{2,1}(\mathbb{T}))\cap L^\infty([0,\infty);V)$, $\mathbb{P}$-a.s., i.e., $
\int_0^\infty\|u(t)\|^2_{W^{2,1}(\mathbb{T})}dt+\sup_{t\in[0,\infty)}\|u(t)\|_{V}<\infty$, $\mathbb{P}$-a.s.
\end{remark}

The proof of Theorem \ref{c2th2} is mainly divided into two steps. Firstly, we prove the existence and uniqueness of strong solutions to \eqref{c2eq6} in Section \ref{wellposedness}. Secondly, we show the regularity result and long-time behavior \eqref{c2eq9}, see Proposition \ref{s6th6+1} and Proposition \ref{s6th2025} in Section \ref{longtime}.

For convenience, from now on, we write
$$F(z,h):=\Phi(1,z,h)-h,\;~~~G(z,h):=\Phi(1,z,h)-h-\epsilon z\partial_x h, \quad \forall z\in\mathbb{R},~ \forall h\in V,$$
and introduce the following operator $\tilde{A}:V\to V^*$, i.e.,
$$\tilde{A}h:=Ah+\int_{\{|z|\leq1\}} G(z,h)\nu(dz).$$
Clearly, if $h\in V$, we then have $[F(z,h)]=[G(z,h)]=0$, and
\begin{equation*}\label{s6eq2}
||F(z,h)||_V\leq ||\partial_x \Phi(1,z,h)||_{L^2}+||h||_V\leq2||h||_V, \quad \forall z\in\mathbb{R}.
\end{equation*}

\section{Well-posedness}\label{wellposedness}
\setcounter{equation}{0}
 \setcounter{definition}{0}

We will only prove the well-posedness of \eqref{c2eq6} without large jumps as following:
\begin{equation}\label{s6eq1}
\begin{cases}
d\tilde{u}(t)=\tilde{A}\tilde{u}(t)dt+\int_{\{|z|\leq1\}} F(z,\tilde{u}(t-))\tilde{N}(dzdt),\quad \forall t>0,\\
\tilde{u}(0)=u_0\in V.
\end{cases}
\end{equation}
For the well-posedness of \eqref{c2eq6}, it is then easy to obtain by using the standard interlacing procedure as in e.g., \cite[section 4.2]{MR3158475}. Since the treatments are similar, we omit the proof of this part.

\begin{proposition}\label{s6th1}
Let $u_0\in V$, then \eqref{s6eq1} admits a unique strong solution $\tilde{u}:=\{\tilde{u}(t)\}_{t\geq0}$ in the sense of Definition \ref{s2 df1} (without large jumps and replace $u$ by $\tilde{u}$).
\end{proposition}
\begin{proof}
The idea is to verify that hypothesis (H0)$\sim$(H4) (stated below) hold, then by combining similar arguments as in \cite[Theorem 2.3]{MR2888287} with \cite[Theorem 1.2]{MR3158475}, we can get Proposition \ref{s6th1} as claimed. Since (H0) is satisfied obviously, the verifications of (H1)$\sim$(H4) require substantial technical work owing to the structure of \eqref{s6eq1}, and the remaining steps are more about similar calculations as in the references mentioned above, we will focus only on the verifications of (H1)$\sim$(H4).\\
(H0) $\{e_k\}_{k\in\mathbb{Z}\backslash\{0\}}$ defined as in \eqref{ek} is an orthonormal basis of $H$, which is also an orthogonal basis of $V$.\\
(H1)(Hemicontinuity) For any $u,v,w\in V$,
$$\lambda\mapsto ~_{V^*}{\left<\tilde{A}(u+\lambda v),w\right>}{_V}, \text{is continuous on $\mathbb{R}$.}$$
(H2)(Local monotonicity) For any $u,v\in V$,
\begin{equation*}
2~_{V^*}{\left<\tilde{A}u-\tilde{A}v,u-v\right>}{_V}+\int_{\{|z|\leq1\}}||F(z,u)-F(z,v)||_H^2 \nu(dz)\leq 0.
\end{equation*}
(H3)(Lyapunov condition) For any $n\in\mathbb{N}$ and $u\in H^n$, we have that $\tilde{A}u\in V$ and
\begin{equation*}
2\left<\tilde{A}u,u\right>_V+\int_{\{|z|\leq1\}}||F(z,u)||_V^2\nu(dz)\leq 0.
\end{equation*}
(H4)(Boundedness) For any $u\in V$, there exists a positive constant $C$ such that
\begin{equation}\label{s6eq22} ||\tilde{A}u||_{V^*}^2+\int_{\{|z|\leq1\}} ||F(z,u)||_H^2\nu(dz)\leq C||u||_V^2.
\end{equation}

Below are the verifications of (H1)$\sim$(H4).\\
\emph{Verification of (H1):} To prove hemicontinuity as claimed above, it suffices to prove that for any $u,v,w\in V$,
$$\lambda\mapsto\int_{\{|z|\leq1\}} ~_{V^*}{\left<G(z,u+\lambda v),w\right>}{_V}\nu(dz)$$
is continuous at $\lambda=0$. From \eqref{c2eq2} we have that for any $u\in V$,
\begin{align}\label{s6eq3}
G(z,u)
&=\epsilon z \int_0^1 [\partial_x \Phi(\theta,z,u) - \partial_x u] d\theta= (\epsilon z)^2 \int_0^1 \int_0^{\theta} \partial_{xx} \Phi(\eta,z,u) d\eta d\theta \nonumber\\
&=\epsilon^2 z^2 \int_0^1 (1-\eta)\partial_{xx} \Phi(\eta,z,u)d\eta
\end{align}
holds in $V^*$.
Therefore, for any $\delta>0$, $|\lambda|\leq\delta$ and $z\in\mathbb{R}$, we have
\begin{align*}
|_{V^*}{\left<G(z,u+\lambda v),w\right>}{_V}|
&=\epsilon^2 z^2 |\int_0^1 (1-\eta)\left<-\partial_x\Phi(\eta,z,u+\lambda v), \partial_x w\right>_{L^2} d\eta|\\
&\leq C z^2 ||w||_V \sup_{\eta\in[0,1]}||\Phi(\eta,z,u+\lambda v)||_V\leq C z^2 ||w||_V \cdot ||u+\lambda v||_V\\
&\leq C z^2 ||w||_V (||u||_V+\delta ||v||_V).
\end{align*}
Meanwhile, by Lemma \ref{c2th1} we have that $\Phi(1,z,u+\lambda v)\to\Phi(1,z,u)$ in $V$ as $\lambda\to0$.
Hence $G(z,u+\lambda v)$ converges to $G(z,u)$ in $V^*$ as $\lambda\to0$.
Hemicontinuity then follows from the dominated convergence theorem.\\
\emph{Verification of (H2):} For any $u,v\in V$, we have
\begin{align}\label{s6eq6}
~_{V^*}{\left<\tilde{A}u-\tilde{A}v,u-v\right>}{_V}=&-\left<\arctan \partial_x u-\arctan \partial_x v, \partial_x u-\partial_x v\right>_H\nonumber\\
&+\int_{\{|z|\leq1\}}~_{V^*}{\left<G(z,u)-G(z,v),u-v\right>}{_V}\nu(dz).
\end{align}
Note that
\begin{align}\label{s6eq8}
&||F(z,u)-F(z,v)||_{L^2}^2\nonumber\\
=&||\Phi(1,z,u)-\Phi(1,z,v)||_{L^2}^2+||u-v||_H^2-2\left<\Phi(1,z,u)-\Phi(1,z,v),u-v\right>_{L^2},
\end{align}
and
\begin{align}\label{s6eq9}
&~_{V^*}{\left<G(z,u)-G(z,v),u-v\right>}{_V}\nonumber\\
=&\left<\Phi(1,z,u)-\Phi(1,z,v),u-v\right>_{L^2}-||u-v||_H^2-\epsilon z\left<\partial_x u-\partial_x v,u-v\right>_{L^2} \nonumber\\
=& \left<\Phi(1,z,u)-\Phi(1,z,v),u-v\right>_{L^2}-||u-v||_H^2,
\end{align}
where in the second equality of \eqref{s6eq9} we used the fact that $\left<\partial_x u-\partial_x v,u-v\right>_{L^2}=0$, which is justified by integration by parts and periodic boundary conditions.
Combining \eqref{s6eq8} with \eqref{s6eq9} yields,
\begin{align}\label{s6eq10}
&||F(z,u)-F(z,v)||_{L^2}^2+2~_{V^*}{\left<G(z,u)-G(z,v),u-v\right>}{_V}\nonumber\\
=&||\Phi(1,z,u)-\Phi(1,z,v)||^2_{L^2}-||u-v||_H^2
\nonumber\\
=&2\epsilon z \int_0^1\left<\partial_x \Phi(\theta,z,u)-\partial_x \Phi(\theta,z,v),\Phi(\theta,z,u)-\Phi(\theta,z,v)\right>_{L^2}d\theta=0,
\end{align}
where in the last line we again used integration by parts and periodic boundary conditions.
Hence, we have
\begin{equation}\label{s6eq13}
||F(z,u)-F(z,v)||_{L^2}^2+2~_{V^*}{\left<G(z,u)-G(z,v),u-v\right>}{_V}=0.
\end{equation}
Now we conclude from \eqref{s6eq6} and \eqref{s6eq13} that
\begin{align}\label{s6eq14}
&2~_{V^*}{\left<\tilde{A}u-\tilde{A}v,u-v\right>}{_V}+\int_{\{|z|\leq1\}}||F(z,u)-F(z,v)||_H^2 \nu(dz)\nonumber\\
=& -2\left<\arctan \partial_x u-\arctan \partial_x v, \partial_x u-\partial_x v\right>_{L^2}\leq 0.
\end{align}
\emph{Verification of (H3):} Note that for any $u\in H^n$,  $n\in\mathbb{N}$,
\begin{equation}\label{s6eq16}
||F(z,u)||^2_V =||\partial_x F(z,u)||^2_H =||\partial_x \Phi(1,z,u)||_H^2+||\partial_x u||_H^2 -2\left<\partial_x \Phi(1,z,u),\partial_x u\right>_H,
\end{equation}
\begin{equation*}\label{s6eq17}
\left<\tilde{A}u,u\right>_V=-\left<\frac{\partial_{xx} u}{1+(\partial_x u)^2}, \partial_{xx} u\right>_H+\int_{\{|z|\leq1\}}\left<G(z,u),u\right>_V\,\nu(dz),
\end{equation*}
\begin{equation}\label{s6eq18} \left<G(z,u),u\right>_V=\left<\partial_x \Phi(1,z,u),\partial_x u\right>_H
-||\partial_x u||_H^2-\epsilon z \left<\partial_{xx}u,\partial_x u\right>_H.
\end{equation}
Similarly to \eqref{s6eq10}, it follows from \eqref{s6eq16} and \eqref{s6eq18} that
\begin{align} \label{s6eq19} &2\left<G(z,u),u\right>_V+||F(z,u)||_V^2
=||\partial_x \Phi(1,z,u)||_H^2-||\partial_x u||_H^2-2\epsilon z \left<\partial_{xx}u,\partial_x u\right>_H=0,
\end{align}
where we used Lemma \ref{c2th1} and the fact that $u\in H^n$.
Hence,
\begin{equation}\label{s6eq21}  2\left<\tilde{A}u,u\right>_V+\int_{\{|z|\leq1\}}||F(z,u)||_V^2\nu(dz)=
-2\left<\frac{\partial_{xx} u}{1+(\partial_x u)^2}, \partial_{xx} u\right>_H\leq0.
\end{equation}
\emph{Verification of (H4):} By \eqref{c2eq2}, it is easy to see that for any $u\in V$,
\begin{align}\label{s6eq23}
||F(z,u)||_H^2
\leq\epsilon^2 z^2 \int_0^1 ||\Phi(\theta,z,u)||_V^2\,d\theta
\leq \epsilon^2 z^2 ||u||_V^2.
\end{align}
Meanwhile, for any $w\in V$ with $||w||_V=1$,
\begin{align}\label{s6eq24}
&\Big|~_{V^*}{\left<\tilde{A}u,w\right>}{_V}\Big|
\leq|\left<\arctan \partial_x u,\partial_x w\right>_{L^2}|
+\Big|\int_{\{|z|\leq1\}}~_{V^*}{\left<G(z,u),w\right>}{_V}\nu(dz)\Big|\nonumber\\
\leq& ||\arctan\partial_x u||_{L^2} \cdot ||\partial_x w||_{L^2}+\epsilon^2\int_{\{|z|\leq1\}}\int_0^1 z^2(1-\eta)\big|\left<-\partial_x\Phi(\eta,z,u),
\partial_x w\right>_{L^2}\big|\, d\eta\nu(dz)\nonumber\\
\leq& ||\partial_x u||_{L^2}+\epsilon^2\int_{\{|z|\leq1\}} z^2 \sup_{\eta\in[0,1]}||\Phi(\eta,z,u)||_V\,\nu(dz)
\leq\left(1+\epsilon^2 \int_{\{|z|\leq1\}}|z|^2\nu(dz)\right)||u||_V,
\end{align}
where we used \eqref{s6eq3} in the second step and Lemma \ref{c2th1} in the last step. Now, \eqref{s6eq22} follows from \eqref{s6eq23} and \eqref{s6eq24} by taking $C:=\left(1+\epsilon^2 \int_{\{|z|\leq1\}}|z|^2\nu(dz)\right)^2+\epsilon^2 \int_{\{|z|\leq1\}}|z|^2\nu(dz) $.

The proof of Proposition \ref{s6th1} is complete.
\end{proof}

\section{Regularity and long-time behavior}\label{longtime}
\setcounter{equation}{0}
 \setcounter{definition}{0}

As stated in the introduction, before  showing the regularity and long-time behavior result of \eqref{c2eq6} (cf. Proposition \ref{s6th6+1} and Proposition \ref{s6th2025}, respectively), we need some boundedness and convergence (cf. Lemma \ref{s6th3}-Lemma \ref{s6th5} below) of the strong solutions $\{u^{(n)}\}_{n\in\Bbb{N}}$ to the following auxiliary equations:
\begin{equation}\label{s6eq25}
\left\{\begin{aligned}
&du^{(n)}(t)= \Pi_n \tilde{A}u^{(n)}(t)dt
+\int_{\{|z|\leq1\}}F(z,u^{(n)}(t-))\tilde{N}(dzdt)\\
&~~~~~~~~~~~~~+\int_{\{|z|>1\}}F(z,u^{(n)}(t-))N(dzdt),\quad \forall t>0,\\
&u^{(n)}(0)=\Pi_n u_0\in H^n,
\end{aligned}\right.
\end{equation}
where $\Pi_n: V^*\to H^n$ is the orthogonal projection and $u_0\in V$. We remark that if $u^{(n)}(t-)\in H^n$, then $F(z,u^{(n)}(t-))\in H^n$, since $ \Phi(1,z,u^{(n)}(s-))=u^n(s-,\cdot+\epsilon z)$
still belongs to the finite dimensional space $H^n$.

\begin{lemma}\label{s6th3} Let $u_0\in V$. Then for the strong solutions $\{u^{(n)}\}_{n\in\Bbb{N}}$ to \eqref{s6eq25}, we have
that for any $t\geq0$,
\begin{equation}\label{s6eq26} ||u^{(n)}(t)||_V^2
+2\int_0^t \int_\mathbb{T} \frac{(\partial_{xx} u^{(n)}(s))^2}{1+(\partial_x u^{(n)}(s))^2}dxds
= ||\Pi_n u_0||_V^2,~~~\mathbb{P}\text{-}a.s.
\end{equation}
\end{lemma}
\begin{proof} Applying It\^o's formula (cf. \cite[Pages 66-67]{MR1011252} ) to $||u^{(n)}(t)||_V^2$, we have that for any $t\geq0$,
\begin{align*}\label{s6eq27}
||u^{(n)}(t)||_V^2=&||\Pi_n u_0||_V^2+2\int_0^t \left<\tilde{A}u^{(n)}(s), u^{(n)}(s)\right>_V\,ds
+\int_0^t \int_{\{|z|\leq1\}} ||F(z,u^{(n)}(s))||_V^2 \,\nu(dz)ds\\
&+\int_0^t\int_{\{|z|\leq1\}} ||F(z,u^{(n)}(s-))+u^{(n)}(s-)||_V^2-||u^{(n)}(s-)||_V^2 \tilde{N}(dzds)\\
&+\int_0^t\int_{\{|z|>1\}} ||F(z,u^{(n)}(s-))+u^{(n)}(s-)||_V^2-||u^{(n)}(s-)||_V^2 N(dzds).
\end{align*}
We then get \eqref{s6eq26} by \eqref{s6eq21} and the fact that
\begin{eqnarray}\label{s6eq28}
||F(z,u^{(n)}(s-))+u^{(n)}(s-)||_V=||\Phi(1,z,u^{(n)}(s-))||_{V}=||u^{(n)}(s-)||_{V}.
\end{eqnarray}
\end{proof}

\begin{lemma}\label{s6th4}
Let $u_0\in V$. Then for the strong solution $\{u(t)\}_{t\geq0}$ to \eqref{c2eq6}, and the strong  solutions $\{u^{(n)}\}_{n\in\Bbb{N}}$ to \eqref{s6eq25}, we have that for any $t\geq0$,
\begin{equation}\label{s6eq30-0.1} \lim_{n\to\infty} ||u(t)-u^{(n)}(t)||_H=0,\;\mathbb{P}\text{-}a.s.,
\end{equation}
\begin{equation}\label{s6eq30-0.2}
\lim_{n\to\infty}\int_0^t\left<\arctan \partial_x u(s)-\arctan \partial_x u^{(n)}(s),\partial_x u(s)-\partial_x u^{(n)}(s)\right>_{L^2} ds=0,\ \mathbb{P}\text{-}a.s.
\end{equation}
\end{lemma}

\begin{proof}
Applying the $\rm It\hat{o}$ formula to $||u(t)-u^{(n)}(t)||^2_H$, we have that for any $t\geq0$,
\begin{align}\label{s6eq30}
&||u(t)-u^{(n)}(t)||^2_H\nonumber\\
=&||(1-\Pi_n)u_0||_H^2+2\int_0^t\ _{V^*}{\left<\tilde{A}u(s)-\Pi_n \tilde{A}u^{(n)}(s),u(s)-u^{(n)}(s)\right>}{_V}ds\nonumber\\
&+\int_0^t \int_{\{|z|\leq 1\}}||F(z,u(s-))-F(z,u^{(n)}(s-))+u(s-)-u^{(n)}(s-)||_H^2-||u(s-)-u^{(n)}(s-)||_H^2 \tilde{N}(dzds)\nonumber\\
&+\int_0^t \int_{\{|z|>1\}} ||F(z,u(s-))-F(z,u^{(n)}(s-))+u(s-)-u^{(n)}(s-)||_H^2-||u(s-)-u^{(n)}(s-)||_H^2 N(dzds)\nonumber\\
&+\int_0^t \int_{\{|z|\leq1\}}  ||F(z,u(s))-F(z,u^{(n)}(s))||_H^2 \nu(dz)ds.
\end{align}
Similarly to \eqref{s6eq28}, we have
\begin{align}\label{eq202502}
& ||F(z,u(s-))- F(z,u^{(n)}(s-))+u(s-)-u^{(n)}(s-)||_H^2\nonumber\\
=&||\Phi(1,z,u(s-))-\Phi(1,z,u^{(n)}(s-))||_H^2
=||u(s-)-u^{(n)}(s-)||_H^2.
\end{align}
Note that
\begin{align}\label{eq202502-Zhai}
&~_{V^*}{\left<(1-\Pi_n)\tilde{A}u^{(n)}(s),u(s)-u^{(n)}(s)\right>}{_V}\nonumber\\
=&~_{V^*}{\left<\tilde{A}u^{(n)}(s),(1-\Pi_n)(u(s)-u^{(n)}(s))\right>}{_V}\nonumber\\
=&~_{V^*}{\left<\tilde{A}u(s),(1-\Pi_n)(u(s)-u^{(n)}(s))\right>}{_V}
-~_{V^*}{\left<\tilde{A}u(s)-\tilde{A}u^{(n)}(s),(1-\Pi_n)(u(s)-u^{(n)}(s))\right>}{_V}
\nonumber\\
=&
~_{V^*}{\left<(1-\Pi_n)\tilde{A}u(s),u(s)-u^{(n)}(s)\right>}{_V}
-
~_{V^*}{\left<\tilde{A}u(s)-\tilde{A}u^{(n)}(s),(1-\Pi_n)u(s)\right>}{_V},
\end{align}
where in the last equality we used the fact that $(1-\Pi_n)u^{(n)}(s)=0$.
Hence, by \eqref{s6eq30}, \eqref{eq202502} and \eqref{eq202502-Zhai}, we have that for any $t\geq0$,
\begin{align}\label{s6eq31}
&||u(t)-u^{(n)}(t)||^2_H
\leq||(1-\Pi_n)u_0||_H^2+2\int_0^t ~_{V^*}{\left<\tilde{A}u(s)-\tilde{A}u^{(n)}(s),u(s)-u^{(n)}(s)\right>}{_V}ds\nonumber\\
&+2\int_0^t ~_{V^*}{\left<(1-\Pi_n)\tilde{A}u(s),u(s)-u^{(n)}(s)\right>}{_V}ds
-2\int_0^t ~_{V^*}{\left<\tilde{A}u(s)-\tilde{A}u^{(n)}(s),(1-\Pi_n)u(s)\right>}{_V}ds\nonumber\\
&+\int_0^t \int_{\{|z|\leq1\}}  ||F(z,u(s))-F(z,u^{(n)}(s))||_H^2 \nu(dz)ds.
\end{align}
By \eqref{s6eq31}, \eqref{s6eq14} and H$\rm\ddot{o}$lder's inequality, we get that for any $t\geq0$,
\begin{align}\label{s6eq32}  &||u(t)-u^{(n)}(t)||^2_H
+2\int_0^{t}\left<\arctan \partial_x u(s)-\arctan \partial_x u^{(n)}(s),\partial_x u(s)-\partial_x u^{(n)}(s)\right>_{L^2} ds\nonumber\\
\leq&||(1-\Pi_n) u_0||_H^2+2\left(\int_0^{t} ||(1-\Pi_n)\tilde{A}u(s)||_{V^*}^2\,ds\right)^{\frac{1}{2}}\left(\int_0^{t}||u(s)-u^{(n)}(s)||_V^2\,ds\right)^{\frac{1}{2}}\nonumber\\
&+2\left(\int_0^t ||(1-\Pi_n)u(s)||_V^2 ds\right)^{\frac{1}{2}}\left(\int_0^{t}||\tilde{A}u(s)-\tilde{A}u^{(n)}(s)||_{V^*}^2\,ds\right)^{\frac{1}{2}}.
\end{align}

By \eqref{s6eq26} and some classical arguments (see e.g. \cite[Chapter 4]{MR2329435}), we know that $\{u^{(n)}\}_{n\in\mathbb{N}}$ admits a subsequence $\{u^{(n_k)}\}_{k\in\mathbb{N}}$ such that as $k\rightarrow\infty$, $u^{(n_k)}\xrightarrow{w}u$ in $L^2([0,t]\times\Omega;H)$. Furthermore, by \eqref{s6eq26} and  \eqref{s6eq22} we have that for any $t\geq0$,
\begin{equation*}\label{20250603}
\mathop{{\rm ess}\,\sup}\limits_{s\in[0,t]}\mathbb{E}||u(s)||_V^2<\infty,\quad
\mathbb{E}\int_0^t||\tilde{A}u(s)||_{V^*}^2 ds<\infty,
\end{equation*}
and consequently we have that for any $t\geq0$,
\begin{equation}\label{20250603+1} \int_0^t ||u(s)||_V^2 ds<\infty,~\int_0^t ||\tilde{A}u(s)||_{V^*}^2 ds<\infty,\quad\mathbb{P}\text{-}a.s.
\end{equation}

By \eqref{20250603+1}, \eqref{s6eq26} and \eqref{s6eq22} again, we have that for any $t\geq0$,
$${\sup_{n}}\int_0^{t}||u(s)-u^{(n)}(s)||_V^2\,ds<\infty,\quad
\sup_{n}\int_0^{t}||\tilde{A}u(s)-\tilde{A}u^{(n)}(s)||_{V^*}^2\,ds<\infty, \,\mathbb{P}\text{-}a.s.$$
Now letting $n\to\infty$ on both sides of \eqref{s6eq32}, by the dominated convergence theorem, we then derive \eqref{s6eq30-0.1} and \eqref{s6eq30-0.2}.
\end{proof}

\begin{lemma}\label{s6th5}
Let $u_0\in V$. For the strong  solution $u:=\{u(t)\}_{t\geq0}$ to \eqref{c2eq6} and the strong solutions $\{u^{(n)}\}_{n\in\Bbb{N}}$ to \eqref{s6eq25}, we have that\\
{\rm (i)} For any $t\geq0$,
$$u^{(n)}(t)\xrightarrow{w}u(t)~\text{in}~ V,~~ \mathbb{P}\text{-}a.s.,\quad\text{and}\quad
||u(t)||_V\leq ||u_0||_V, ~~ \mathbb{P}\text{-}a.s.$$
{\rm (ii)}  Let $f(x)=\log(x+\sqrt{1+x^2})$, $\forall x\in\Bbb{R}$, then for any fixed $T>0$,
$$\partial_x \Big(f(\partial_x u^{(n)})\Big)\xrightarrow{w}\partial_x \Big(f(\partial_x u)\Big)~\text{in}~ L^2([0,T]\times \mathbb{T}),~~\mathbb{P}\text{-}a.s.$$
{\rm (iii)} For any fixed $T>0$,
\begin{equation}\label{20250601}
||u(T)||_V^2+\int_0^T\int_\mathbb{T}\frac{(\partial_{xx} u(s))^2}{1+(\partial_x u(s))^2}dxds\leq||u_0||_V^2,~~\mathbb{P}\text{-}a.s.
\end{equation}
\end{lemma}
\begin{proof}  To avoid confusion, we remark that all the convergences and estimations in the following proof hold $\mathbb{P}$-a.s.

\noindent {\rm (i)} For any $t\geq0$, from \eqref{s6eq26} we know that $\{u^{(n)}(t)\}_{n\in\mathbb{N}}$ is uniformly bounded in $V$, so for any sequence $\{n_k\}_{k\in\mathbb{N}}\subset\{n\}_{n\in\mathbb{N}}$, $\{u^{(n_k)}(t)\}_{k\in\mathbb{N}}$ has a weakly convergent subsequence (still denoted by $\{n_k\}_{k\in\mathbb{N}}$) in $V$. By \eqref{s6eq30-0.1} and the uniqueness of limit we know that as $k\rightarrow\infty$, $u^{(n_k)}(t)\xrightarrow{w} u(t)$ in $V$. By the arbitrariness of $\{n_k\}_{k\in\mathbb{N}}$ we get that as $n\rightarrow\infty$, $u^{(n)}(t)\xrightarrow{w} u(t)$ in $V$, and
$$||u(t)||_V^2\leq \liminf_{n\to\infty}||u^{(n)}(t)||_V^2\leq ||u_0||_V^2.$$

\noindent {\rm (ii)}  We will prove this result by using the properties of maximal monotone operators. To this end, we first carry out some preliminary work. Fix $T>0$. By Lemma \ref{s6th4},  similarly to (i), we know that \begin{eqnarray}\label{limit0}
  u^{(n)}\to u ~\text{in}~ L^2([0,T];H),\ \text{as}\ n\rightarrow\infty,
\end{eqnarray}
and
\begin{eqnarray}\label{limit1}
  \partial_x u^{(n)}\xrightarrow{w} \partial_x u ~\text{in}~ L^2([0,T];H),\ \text{as}\   n\rightarrow\infty.
\end{eqnarray}
Since  $\partial_x (f(\partial_x u^{(n)}(s)))=\frac{\partial_{xx} u^{(n)}(s)}{\sqrt{1+(\partial_x u^{(n)}(s))^2}}$, from \eqref{s6eq26} we see that $\{\partial_x (f(\partial_x u^{(n)}(\cdot)))\}_{n\in\mathbb{N}}$ is uniformly bounded in $L^2([0,T]\times\mathbb{T})$, then for any sequence $\{n_k\}_{k\in\mathbb{N}}\subset\{n\}_{n\in\mathbb{N}} $, there exists a subsequence (still denoted by $\{n_k\}_{k\in\mathbb{N}}$) such that as $k\rightarrow\infty$, $\{\partial_x (f(\partial_x u^{(n_k)}(\cdot)))\}_{k\in\mathbb{N}}$ weakly convergent in $L^2([0,T]\times\mathbb{T})$. Meanwhile, since $|f(x)|\leq|x|$, $\forall x\in\Bbb{R}$, there exists a subsequence of $\{n_k\}_{k\in\mathbb{N}}$ (still denoted by $\{n_k\}_{k\in\mathbb{N}}$) and $U\in L^2([0,T]\times\mathbb{T})$ such that
\begin{eqnarray}\label{limit2}
f(\partial_x u^{(n_k)}(\cdot))\xrightarrow{w} U ~\text{in}~ L^2([0,T]\times\mathbb{T}),\ \text{as}\   k\rightarrow\infty.
\end{eqnarray}
Therefore, by the uniqueness of limit we have that as $k\rightarrow\infty$, $\partial_x \Big(f(\partial_x u^{(n_k)}(\cdot))\Big)\xrightarrow{w}\partial_x U$ in $L^2([0,T]\times\mathbb{T})$, and by \eqref{limit0} we have
\begin{equation*}\label{20251001}
\lim_{k\to\infty}\int_0^T \left<u^{(n_k)}(s),\partial_x \Big(f(\partial_x u^{(n_k)}(s))\Big)\right>_{L^2}ds=\int_0^T \left<u(s),\partial_x U(s)\right>_{L^2}dt.
\end{equation*}
By the integration by parts we also have
\begin{equation}\label{20251002}
\lim_{k\to\infty}\int_0^T \left<\partial_x u^{(n_k)}(s),f(\partial_x u^{(n_k)}(s))\right>_{L^2}ds=\int_0^T \left<\partial_x u(s), U(s)\right>_{L^2}ds.
\end{equation}
Note that $f:L^2([0,T]\times\mathbb{T})\to L^2([0,T]\times\mathbb{T})$ is maximal monotone, by \eqref{limit1}, \eqref{limit2}, \eqref{20251002}, and \cite[Corollary 2.4]{MR2582280} we have $U=f(\partial_x u)$. Then, by the arbitrariness of $\{n_k\}_{k\in\mathbb{N}}$ we get $\partial_x \Big(f(\partial_x u^{(n)})\Big)\xrightarrow{w}\partial_x \Big(f(\partial_x u)\Big)$ in $L^2([0,T]\times\mathbb{T})$.\\
\noindent {\rm (iii)} \eqref{20250601} is a direct corollary of \eqref{s6eq26} and the weak convergences (i) (ii).
\end{proof}

The following corollary is a modification of Lemma \ref{s6th5} (iii) and will be used later. Since their proofs are similar, we omit its proof.

\begin{corollary}\label{s6th6} Let $u_0\in V$, $\{u(t)\}_{t\geq0}$ be the strong solution to \eqref{c2eq6}. Then for any  $0\leq s<t$, we have
\begin{equation}\label{20250601+++}||u(t)||_V^2+\int_s^t\int_\mathbb{T}\frac{(\partial_{xx} u(r))^2}{1+(\partial_x u(r))^2}dxdr\leq||u(s)||_V^2,~~\mathbb{P}\text{-}a.s.
\end{equation}
\end{corollary}

The following proposition reveals the regularity result of $u$ and the convergence of $\mathbb{E}||u(t)||_V^2$ as $t\rightarrow\infty$.

\begin{proposition}\label{s6th6+1}
Let $u_0\in V$, $u:=\{u(t)\}_{t\geq0}$ be the strong  solution to \eqref{c2eq6}. Then we have $u\in L^2([0,\infty); W^{2,1}(\mathbb{T}))\cap L^\infty([0,\infty);V)$, $\mathbb{P}$-a.s., and $||u(t)||_V\to0$ as $t\to\infty$, $\mathbb{P}$-a.s.
\end{proposition}
\begin{proof}
From \eqref{20250601}, it is easy to see that $u\in L^\infty([0,\infty);V)$. By Lemma \ref{s6th5} we have that for any fixed $T>0$,
\begin{align}\label{s6eq+1} \int_0^T ||\partial_{xx} u(r)||^2_{L^1}dr
&=\int_0^T \left(\int_\mathbb{T} \frac{|\partial_{xx} u(r)|}{\sqrt{1+(\partial_x u(r))^2}}\cdot \sqrt{1+(\partial_x u(r))^2} dx\right)^2dr\nonumber\\
&\leq\int_0^T \left((1+||u(r)||_V^2)\cdot\int_\mathbb{T} \frac{|\partial_{xx} u(r)|^2}{1+(\partial_x u(r))^2}dx\right)dr\nonumber\\
&\leq(1+||u_0||_V^2)\int_0^T \int_\mathbb{T} \frac{|\partial_{xx} u(r)|^2}{1+(\partial_x u(r))^2}dxdr\nonumber\\
&\leq(1+||u_0||_V^2)||u_0||_V^2<\infty,~~ \mathbb{P}\text{-}a.s.
\end{align}
Note that the control on the right-hand side of \eqref{s6eq+1} does not depend on $T$,  we then conclude that $u\in L^2([0,\infty); W^{2,1}(\mathbb{T}))$, $\mathbb{P}$-a.s.

By the Sobolev embedding we have
$\int_0^\infty ||u(t)||^2_V dt<\infty,$
together with the fact that $||u(t)||^2_V$ is decreasing with respect to $t\geq0$, $\mathbb{P}$-a.s. (see Corollary \ref{s6th6}), we deduce that $||u(t)||^2\to0$ as $t\to\infty$, $\mathbb{P}$-a.s.
\end{proof}

Before proving the long-time behavior of $\{u(t)\}_{t\geq0}$ in $V$, we will first prove its long-time behavior in $H$.
\begin{lemma}\label{s6th20250801} Let $u_0\in V$,  $u:=\{u(t)\}_{t\geq0}$ be the strong  solution to \eqref{c2eq6}. Then there exists a constant $k_1>0$ (independent of $t$ and $\omega$) such that for any $t\geq0$,
\begin{equation}\label{202508002} ||u(t)||^2_H\leq ||u_0||_H^2 e^{-k_1 t},~~\mathbb{P}\text{-}a.s.
\end{equation}
\end{lemma}
\begin{proof}
Applying It\^o's formula to $||u(t)||_H^2$, we have that for any $t\geq0$,
\begin{align*}
||u(t)||_H^2=&||u_0||_H^2+2\int_0^t ~_{V^*}{\left<\tilde{A}u(s), u(s)\right>}{_V}\,ds
+\int_0^t \int_{\{|z|\leq1\}} ||F(z,u(s))||_H^2 \,\nu(dz)ds\\
&+\int_0^t\int_{\{|z|\leq1\}} \big(||F(z,u(s-))+u(s-)||_H^2-||u(s-)||_H^2\big) \,\tilde{N}(dzds)\\
&+\int_0^t\int_{\{|z|>1\}} \big(||F(z,u(s-))+u(s-)||_H^2-||u(s-)||_H^2\big) \,N(dzds)\\
=&||u_0||_H^2-2\int_0^t\int_\mathbb{T}\arctan(\partial_x u(s))\cdot\partial_x u(s) dxds,~~\mathbb{P}\text{-}a.s.,
\end{align*}
where in the last equality we used the fact that $||F(z,u(s-))+u(s-)||_{H}=||\Phi(1,z,u(s-))||_{H}=||u(s-)||_{H}$ and \eqref{s6eq14} with $v\equiv0$. Since for any $s\geq0$,
\begin{align*}
\int_\mathbb{T}\arctan(\partial_x u(s))\cdot\partial_x u(s) dx\leq&\Big(\int_\mathbb{T}|\arctan(\partial_x u(s))|^2dx\Big)^{\frac{1}{2}}
\Big(\int_\mathbb{T}|\partial_x u(s)|^2dx\Big)^{\frac{1}{2}}\\
\leq&\frac{\pi}{2}
\Big(||u(s)||^2_V\Big)^{\frac{1}{2}}
 \leq\frac{\pi}{2}\left(||u_0||_V^2\right)^{\frac{1}{2}}<\infty,
\end{align*}
where \eqref{20250601+++} was used to obtain the last inequality,  we have
\begin{equation}\label{202508004}
\frac{d}{dt}||u(t)||_H^2=
-2\int_\mathbb{T}\arctan(\partial_x u(t))\cdot\partial_x u(t) dx,~~ \mathbb{P}\text{-}a.s.
\end{equation}

By the Sobolev embedding, there exists a positive constant $C_1$ such that $||v||_{H}\leq C_1||\partial_x v||_{L^1},\,\forall v\in V$. Since $g(x):=x\cdot\arctan x$ is convex, non-negative, and
$g(x)\geq\frac{x^2}{1+x^2}$, $\forall x\geq0$,  we have that for any $t\geq0$,
\begin{align}\label{202508005}
\int_\mathbb{T}\arctan(\partial_x u(t))\cdot\partial_x u(t) dx
\geq&
||\partial_x u(t)||_{L^1}\cdot\arctan||\partial_x u(t)||_{L^1}\nonumber\\
\geq&
\frac{||\partial_x u(t)||_{L^1}^2}{1+||\partial_x u(t)||_{L^1}^2}\geq
\frac{||\partial_x u(t)||_{L^1}^2}{1+||\partial_x u(t)||_{L^2}^2}\nonumber\\
\geq&
\frac{1}{C_1^2}\cdot \frac{||u(t)||_{H}^2}{1+||\partial_x u(t)||_{L^2}^2}
\geq \frac{||u(t)||_H^2}{C_1^2(1+||u_0||_{V}^2)},~~\mathbb{P}\text{-}a.s.,
\end{align}
where Lemma \ref{s6th5} (i) was used to obtain the last inequality.

Setting $k_1:=(C_1^2(1+||u_0||_{V}^2))^{-1}$, taking \eqref{202508004} and \eqref{202508005} into account, we conclude that for any $t\geq0$,
\begin{equation}\label{202508008} \frac{d}{dt}||u(t)||_H^2\leq -k_1||u(t)||_H^2,~~\mathbb{P}\text{-}a.s.
\end{equation}
\eqref{202508002} then follows from \eqref{202508008} and Gronwall's inequality.
\end{proof}

Now, we are ready to present the long-time behavior of $\{u(t)\}_{t\geq0}$ in $V$.

\begin{proposition} \label{s6th2025}
Let $u_0\in V$, $u:=\{u(t)\}_{t\geq0}$ be the strong  solution to \eqref{c2eq6}. Then $||u(t)||_V$ converges to 0 with exponential rate as $t\rightarrow\infty$, $\mathbb{P}$-a.s., in the sense that there exists a constant $k_0>0$ (independent of $t$ and $\omega$) such that
\begin{equation*}\label{s6eq39} \lim_{t\to\infty}e^{k_0 t}||u(t)||^2_V=0,~~\mathbb{P}\text{-}a.s.
\end{equation*}
\end{proposition}
\begin{proof} From the proof of \eqref{202508005} and Lemma \ref{s6th5} (i), we see that for any $t\geq0$,
\begin{eqnarray*}\label{202508005 V2}
\int_\mathbb{T}\arctan(\partial_x u(t))\cdot\partial_x u(t) dx
\geq
\frac{||\partial_x u(t)||_{L^1}^2}{1+||\partial_x u(t)||_{L^2}^2}
\geq \frac{||\partial_x u(t)||_{L^1}^2}{1+||u_0||_{V}^2}
,~~ \mathbb{P}\text{-}a.s.
\end{eqnarray*}
By \eqref{202508004} we have that for any fixed  $T>t\geq0$,
\begin{eqnarray*}\label{s6eq27-20250809-01}
||u(T)||_H^2+2\int_t^T\int_\mathbb{T}\arctan(\partial_x u(s))\cdot\partial_x u(s) dxds
=||u(t)||_H^2,~~\mathbb{P}\text{-}a.s.,
\end{eqnarray*}
which implies that $T$ can be equal to $\infty$.
Hence by Lemma \ref{s6th20250801}, for any $t\geq0$,
\begin{align}\label{202508011}
\int_t^\infty ||\partial_x u(s)||_{L^1}^2 ds&\leq
(1+||u_0||_V^2)\int_t^\infty \left(\int_\mathbb{T}\arctan(\partial_x u(s))\cdot\partial_x u(s) dx\right)ds\nonumber\\
&=(1+||u_0||_V^2)||u(t)||_H^2
\leq (1+||u_0||_V^2)||u_0||_H^2 e^{-k_1 t},~~\mathbb{P}\text{-}a.s.
\end{align}

By the Sobolev embedding, there exists a positive constant $C$, which is independent of $t$ and $\omega$, such that $||u(t)||_V^2\leq||u(t)||_{L^\infty}||\partial_{xx} u(t)||_{L^1} \leq C||\partial_x u(t)||_{L^1}||\partial_{xx} u(t)||_{L^1}$. Therefore,
\begin{equation}\label{202508013} \int_t^\infty ||u(s)||_{V}^2 ds\leq C
\left( \int_t^\infty ||\partial_x u(s)||_{L^1}^2 ds\right)^{\frac{1}{2}}\left( \int_t^\infty||\partial_{xx} u(s)||_{L^1}^2 ds\right)^{\frac{1}{2}},~~\mathbb{P}\text{-}a.s.
\end{equation}
Note that by \eqref{s6eq+1}, for any $t\geq0$,
\begin{equation}\label{202508014}\int_t^\infty ||\partial_{xx} u(s)||_{L^1}^2 ds\leq (1+||u_0||_V^2)||u_0||_V^2<\infty,~~\mathbb{P}\text{-}a.s.
\end{equation}
Thus, taking \eqref{202508011}$\sim$\eqref{202508014} into account, we have that for any $t\geq0$,
\begin{equation*}\label{202508016}
\int_t^\infty ||u(s)||_{V}^2 ds\leq C (1+||u_0||_V^2)||u_0||_V||u_0||_H e^{\frac{-k_1 t}{2}}, ~~\mathbb{P}\text{-}a.s.
\end{equation*}
Setting $k_0:=k_1/4$, then by L'H$\rm\hat{o}$pital's Rule,
 \begin{equation*}\label{202508017} \lim_{t\to\infty}\frac{||u(t)||_{V}^2}{e^{-k_0 t}}
=\lim_{t\to\infty} \frac{\int_t^\infty ||u(s)||_{V}^2 ds}{k_0^{-1}e^{-k_0 t}}
\leq k_0{C}(1+||u_0||_V^2)||u_0||_V||u_0||_H\left(\lim_{t\to\infty}e^{-k_0 t}\right)=0,
\end{equation*}
which completes the proof of Proposition \ref{s6th2025}.
\end{proof}

Now we conclude from Proposition \ref{s6th1}, Corollary \ref{s6th6+1} and Proposition \ref{s6th2025} that the proof of Theorem \ref{c2th2} is complete.

\section*{Acknowledgements} This work is partially supported by National Key R\&D
Program of China(No. 2022YFA1006001).
Shijie Shang is supported by the National Natural Science Foundation of China (NSFC) (No.12571158), and the Fundamental Research Funds for the Central Universities (No. WK0010000081). Weina Wu is supported by NSFC (No.11901285). Jianliang Zhai is supported by NSFC (No.12371151, 12131019), and the Fundamental Research Funds for the Central Universities (No. WK3470000031, WK0010000081).


\begin{thebibliography}{99}

\bibitem{MR1225209}
Luis Alvarez, Fr\'{e}d\'{e}ric Guichard, Pierre-Louis Lions, Jean-Michel Morel.
\newblock Axioms and fundamental equations of image processing.
\newblock {\em Arch. Rational Mech. Anal.}, 123, no. 3, 199-257, 1993.

\bibitem{MR2582280}
Viorel Barbu.
\newblock {\em Nonlinear Differential Equations of Monotone Types in Banach Spaces}.
\newblock Springer Monographs in Mathematics. Springer, New York, 2010.

\bibitem{MR3155251}
Giovanni Bellettini.
\newblock {\em Lecture Notes on Mean Curvature Flow, Barriers and Singular Perturbations}. Volume 12 of  Appunti. Scuola Normale Superiore di Pisa (Nuova Serie) [Lecture Notes. Scuola Normale Superiore di Pisa (New Series)].
\newblock Edizioni della Normale, Pisa, 2013.

\bibitem{MR3158475}
Zdzis{\l}aw Brze\'{z}niak, Wei Liu, Jiahui Zhu.
\newblock Strong solutions for SPDE with locally monotone coefficients driven by L\'evy noise.
\newblock {\em Nonlinear Anal. Real World Appl.}, 17:283-310, 2014.

\bibitem{BLZ2021}
Zdzis{\l}aw Brze\'{z}niak, Wei Liu, Jiahui Zhu.
\newblock The stochastic Strichartz estimates and stochastic nonlinear Schr\"{o}dinger equations driven by L\'{e}vy noise.
\newblock {\em J. Funct. Anal.}, 281, no. 4, Paper No. 109021, 37 pp, 2021.


\bibitem{BM2019}
Zdzis{\l}aw Brze\'{z}niak,  Utpal Manna.
\newblock Weak solutions of a stochastic Landau-Lifshitz-Gilbert equation driven by pure jump noise.
\newblock {\em Comm. Math. Phys.}, 371, no. 3, 1071-1129,  2019.

\bibitem{CKS1997}
Vicent Caselles, Ron Kimmel, Guillermo Sapiro.
Geodesic active contours.
\newblock {\em International Journal of Computer Vision,} 22(1), 61-79, 1997.
\newblock https://doi.org/10.1023/A:1007979827043


\bibitem{MR4221662}
Nils Dabrock, Martina Hofmanov\'{a},  Matthias R\"{o}ger.
\newblock Existence of martingale solutions and large-time behavior for a stochastic mean curvature flow of graphs.
\newblock {\em Probab. Theory Related Fields.}, 179(1-2):407-449, 2021.

\bibitem{MR4089785}
Konstantinos Dareiotis, Benjamin Gess.
\newblock Nonlinear diffusion equations with nonlinear gradient noise.
\newblock {\em Electron. J. Probab.}, 25:Paper No. 35, 43, 2020.


\bibitem{MR2888287}
Abdelhadi Es-Sarhir, Max-K. von Renesse.
\newblock Ergodicity of stochastic curve shortening flow in the plane.
\newblock {\em SIAM J. Math. Anal.}, 44(1):224-244, 2012.

\bibitem{MR2996423}
Abdelhadi Es-Sarhir, Max-K. von Renesse, Wilhelm Stannat.
\newblock Estimates for the ergodic measure and polynomial stability of plane stochastic curve shortening flow.
\newblock {\em NoDEA Nonlinear Differential Equations Appl.}, 19(6):663-675, 2012.



\bibitem{FPR2025}
Franco Flandoli, Andrea Papini, Marco Rehmeier.
\newblock Average dissipation for stochastic transport equations with L\'{e}vy noise.
\newblock Conference paper. Stochastic Transport in Upper Ocean Dynamics III, pages 45-59. Springer Nature
Switzerland, Cham, 2025.



\bibitem{G2020} Ed Gallagher.
\newblock {\em Curve Shortening Flow}.
\newblock Phd thesis, Durham University, May 2020.


\bibitem{MR4674013}
Paul Gassiat, Benjamin Gess, Pierre-Louis Lions, Panagiotis E. Souganidis.
\newblock Long-time behavior of stochastic Hamilton-Jacobi equations.
\newblock {\em J. Funct. Anal.}, 286(4):Paper No. 110269, 49, 2024.


\bibitem{MR4692883}
Andris Gerasimovi\v{c}s, Martin Hairer, Konstantin Matetski.
\newblock Directed mean curvature flow in noisy environment.
\newblock {\em Comm. Pure Appl. Math.}, 77(3):1850-1939, 2024.


\bibitem{MR3605963}
Benjamin Gess, Michael R\"{o}ckner.
\newblock Stochastic variational inequalities and regularity for degenerate stochastic partial differential equations.
\newblock {\em Trans. Amer. Math. Soc.}, 369(5):3017-3045, 2017.

\bibitem{MR3205643}
Benjamin Gess, Jonas M. T\"{o}lle.
\newblock Multi-valued, singular stochastic evolution inclusions.
\newblock {\em J. Math. Pures Appl.}, (9)101(6):789-827, 2014.

\bibitem{MR3651056}
Martina Hofmanov\'{a}, Matthias R\"{o}ger, Max von Renesse.
\newblock Weak solutions for a stochastic mean curvature flow of two-dimensional graphs.
\newblock {\em Probab. Theory Related Fields.}, 168(1-2):373-408, 2017.


\bibitem{MR1011252}
Nobuyuki Ikeda, Shinzo Watanabe.
\newblock {\em Stochastic Differential Equations and Diffusion Processes}.
\newblock Second edition. North-Holland Mathematical Library, 24. North-Holland Publishing Co., Amsterdam; Kodansha, Ltd., Tokyo, 1989.





\bibitem{10.1143/PTP.67.147}
Kyozi Kawasaki, Takao Ohta.
\newblock Kinetic drumhead model of interface. I.
\newblock {\em Progress of Theoretical Physics}, 67(1):147-163, 1982.



\bibitem{KPP-95}
Thomas G. Kurtz, \'{E}tienne Pardoux, Philip Protter. \newblock Stratonovich stochastic differential equations driven by general semimartingales.
\newblock {\em Ann. Inst. H. Poincar\'{e} Probab. Statist.}, 31(2), 351-377, 1995.


\bibitem{MR1959710}
Pierre-Louis Lions, Panagiotis E. Souganidis.
\newblock Viscosity solutions of fully nonlinear stochastic partial differential equations.
\newblock Viscosity solutions of differential equations and related topics (Japanese) (Kyoto, 2001). No. 1287, 58-65, 2002.

\bibitem{MR3053478}
Wei Liu.
\newblock Well-posedness of stochastic partial differential equations with Lyapunov condition.
\newblock {\em J. Differential Equations.}, 255(3):572-592, 2013.



\bibitem{LT2025}
Dejun Luo, Feifan Teng.
\newblock Eddy viscosity by L\'{e}vy transport noises, \newblock arXiv:2510.13463, 2025.


\bibitem{MR1343443}
Ravi Malladi, James A Sethian.
\newblock Image processing via level set curvature flow.
\newblock {\em Proc. Nat. Acad. Sci.}, U.S.A. 92, no. 15, 7046-7050, 1995.


\bibitem{MA2021}
Utpal Manna, Akash Ashirbad Panda.
\newblock Well-posedness and large deviations for 2D stochastic constrained Navier-Stokes equations driven by L\'{e}vy noise in the Marcus canonical form.
\newblock {\em J. Differential Equations.}, 302, 64-138, 2021.


\bibitem{MR2815949}
Carlo Mantegazza.
\newblock {\em Lecture Notes on Mean Curvature Flow}. Volume 290 of {\em Progress in Mathematics}.
\newblock Birkh\"{a}user/Springer Basel AG, Basel, 2011.


\bibitem{Marcus-78}
Steven I. Marcus.
\newblock Modeling and analysis of stochastic differential equations driven by point processes.
\newblock {\em IEEE Trans. Inform. Theory.}, 24(2), 164-172, 1978.


\bibitem{MB2020}
Ji\v{r}\'{\i} Minar\v{c}\'{\i}k, Michal Bene\v{s}.
\newblock Long-term behavior of curve shortening flow in $\mathbb{R}^3$.
\newblock {\em SIAM J. Math. Anal.}, 52, no. 2, 1221-1231, 2020.



\bibitem{MR2356959} Szymon Peszat, Jerzy Zabczyk.
\newblock {\em Stochastic Partial Differential Equations with L\'evy noise: An Evolution Equation Approach}. \newblock Encyclopedia Math. Appl., 113
Cambridge University Press, Cambridge, 2007.


\bibitem{MR2329435}
Claudia Pr\'{e}v\^{o}t,  Michael R\"{o}ckner.
\newblock {\em A Concise Course on Stochastic Partial Differential Equations}.
\newblock Lecture Notes in Mathematics, 1905. Springer, Berlin, 2007.

\bibitem{MR4646815}
Florian Seib, Wilhelm Stannat, Jonas M. T\"olle.
\newblock Stability and moment estimates for the stochastic singular $\Phi$-Laplace equation.
\newblock {\em J. Differential Equations.}, 377:663-693, 2023.


\bibitem{ZLZ2026}
Jiahui Zhu, Wei Liu, Jianling Zhai.
\newblock Large deviation principles for stochastic nonlinear Schr\"{o}dinger equations driven by L\'{e}vy noise.
\newblock {\em J. Funct. Anal.}, 290, no. 9, Paper No. 111377, 59 pp, 2026.



\bibitem{MR1931534}
Xiping Zhu.
\newblock {\em Lectures on Mean Curvature Flows}. Volume~32 of {\em AMS/IP Studies in Advanced Mathematics}.
\newblock American Mathematical Society, Providence, RI; International Press, Somerville, MA, 2002.

\end{thebibliography}
\end{document}